\documentclass[a4paper,11pt]{article}
\usepackage{mathrsfs}
\usepackage{latexsym}
\usepackage{amsmath,amssymb}
\usepackage{amsthm}
\usepackage{amsfonts}
\usepackage[usenames]{color}
\usepackage{amssymb}
\usepackage{graphicx}
\usepackage{amsmath}
\usepackage{amsfonts}
\usepackage{amsthm}
\usepackage{mathrsfs}
\usepackage{dsfont}
\usepackage{indentfirst}
\usepackage{enumerate}
\usepackage[square, comma, sort&compress, numbers]{natbib}

\newtheoremstyle{mythm}{1.5ex plus 1ex minus .2ex}{1.5ex plus 1ex
minus .2ex}{\kai}{\parindent}{\song\bfseries}{}{1em}{}
\numberwithin{equation}{section}

\newtheorem{thm}{Theorem}[section]
\newtheorem{lemma}{Lemma}[section]

\newtheorem{corollary}{Corollary}[section]
\allowdisplaybreaks[4]
\newcommand\leqs{\leqslant}
\newcommand\geqs{\geqslant}

\makeatletter
\newcommand\Rmnum[1]{\expandafter\@slowromancap\romannumeral #1@}
\makeatother

\title{On the Hessian Hardy-Sobolev Inequality and Related Variational Problems}
\author{Rongxun He and Wei Ke}

\begin{document}
\date{}
\maketitle

\begin{abstract}
In this paper, we first prove the Hardy-Sobolev inequality for the Hessian integral by means of a descent gradient flow of certain Hessian functionals. As an application, we study the existence and regularity results of solutions to related variational problems. Our results extend the variational theory of the Hessian equation in \cite{CW01variational}.
\end{abstract}

\section{Introduction}
For a smooth function $u$, let $\lambda(D^2u)=(\lambda_1,\ldots,\lambda_n)$ be the eigenvalues of the Hessian matrix $D^2u$. Define the $k$-Hessian operator $S_k$($1\leqs k\leqs n$)
\begin{align*}
S_k(D^2u):=\sigma_k(\lambda(D^2u))=\sum_{i_1<\cdots<i_k}\lambda_{i_1}\cdots\lambda_{i_k}.
\end{align*}
Here, $\sigma_k(\lambda)$ denotes the $k$-th elementary symmetric polynomial of $\lambda$. Alternatively, $S_k(D^2u)$ equals the sum of the principal minors of order $k$ for $D^2u$. According to \cite{CNS85Dirichlet}, we call a function $u\in C^2$ to be $k$-admissible, if $\lambda(D^2u)$ belongs to the symmetric G\aa rding cone $\Gamma_k$, which is given by
\begin{align*}
\Gamma_k=\{\lambda\in\mathbb{R}^n:\sigma_j(\lambda)>0,\ j=1,\ldots,k\}.
\end{align*}
Given a bounded domain $\Omega\subset\mathbb{R}^n$, we denote by $\Phi^k(\Omega)$ the set of all $k$-admissible functions defining on $\Omega$ and by $\Phi_0^k(\Omega)$ the set of all $k$-admissible functions vanishing on the boundary $\partial\Omega$. We call a bounded domain $\Omega$ of class $C^2$ to be strictly $(k-1)$-convex, if there exists a positive constant $K$ such that for every $x\in\partial\Omega$,
\begin{align*}
(\kappa_1(x),\ldots,\kappa_{n-1}(x),K)\in\Gamma_k,
\end{align*}
where $\kappa_1(x),\cdots,\kappa_{n-1}(x)$ denote the principal curvatures of $\partial\Omega$ at $x$. In this paper, we always assume that $\Omega$ is strictly $(k-1)$-convex.

In \cite{Wang94class}, Wang studied the functional $I_k(u)$ given by
\begin{align*}
I_k(u):=\int_\Omega(-u)S_k(D^2u)dx,
\end{align*}
and verified that $\Vert u\Vert_{\Phi_0^k(\Omega)}=[I_k(u)]^{1/(k+1)}$ is a norm in $\Phi_0^k(\Omega)$. Additionally, Wang derived the following Sobolev-type inequality:
\begin{align}
\Vert u\Vert_{L^p(\Omega)}\leqs C\Vert u\Vert_{\Phi_0^k(\Omega)} \quad\text{holds for all }u\in\Phi_0^k(\Omega),
\label{formu11}
\end{align}
where $p\in[1,k^{\star}]$. Here, $k^{\star}$ is the critical exponent for $k$-Hessian operator,
\begin{align*}
k^{\star}\left\{
\begin{array}{ll}
 =\frac{n(k+1)}{n-2k}  & \text{if }2k<n, \\
 <\infty    & \text{if }2k=n, \\
 =\infty    & \text{if }2k>n.
\end{array}
\right.
\end{align*}
Moreover, Tian and Wang proved a Moser-Trudinger type inequality for the case $2k=n$ in \cite{TW10MoserTrudinger}, that is
\begin{align}
\sup\left\{\int_\Omega\exp\bigg[\alpha_n\Big(\frac{|u|}{\Vert u\Vert_{\Phi_0^k(\Omega)}}\Big)^{(n+2)/n}\bigg]dx:u\in\Phi_0^k(\Omega)\right\}\leqs C,
\label{formu12}
\end{align}
where $\alpha_n=n\big[\frac{\omega_n}{k}\binom{n-1}{k-1}\big]^{2/n}$ and $\omega_n$ denotes the area of the unit sphere in $\mathbb{R}^n$.

In this paper, we will utilize the idea of \cite{Wang94class} and obtain a Hardy-Sobolev-type inequality related to the Hessian integral $I_k(u)$. Before that, we first denote a weighted $L^p$-norm:
\begin{align*}
\Vert u\Vert_{L^p(\Omega;|x|^{\sigma})} =\left(\int_\Omega|x|^{\sigma}|u|^pdx\right)^{1/p}.
\end{align*}
Then given any $u\in C_0^\infty(\mathbb{R}^n)$, the classical Hardy-Sobolev inequality is stated as
\begin{align*}
\Vert u\Vert_{L^p(\mathbb{R}^n;|x|^{\sigma})}\leqs C_q\Vert Du\Vert_{L^q(\mathbb{R}^n)},
\end{align*}
where $1\leqs q<n$, $-q\leqs\sigma\leqs0$ and $p=q(n+\sigma)/(n-q)$.
Note that the previous inequality includes the Sobolev inequality as $\sigma=0$ and the Hardy's inequality as $\sigma=-q$. The first main result of our article is as follows.

\begin{thm}[Hessian Hardy-Sobolev inequality]\label{thm11}
Let $\Omega\subset\mathbb{R}^n$ be any smooth $(k-1)$-convex domain containing the origin. Suppose that $n>2k$, $-1\leqs s\leqs0$ and $k^*=k^*(s)>0$ be such that
\begin{align}
k^*=\frac{(k+1)(n+2sk)}{n-2k}.
\label{formu13}
\end{align}
Then it holds for all $u\in\Phi_0^k(\Omega)$,
\begin{align}
\Vert u\Vert_{L^{k^*}(\Omega;|x|^{2sk})}\leqs C\Vert u\Vert_{\Phi_0^k(\Omega)},
\label{formu14}
\end{align}
where the constant $C$ depends only on $n,k$ and $s$. In particular, if $-1<s\leqs0$, the best constant can be attained when $\Omega=\mathbb{R}^n$ by the function
\begin{align}
u(x)=-(\lambda+|x|^{2(s+1)})^{(2k-n)/2k(s+1)}
\label{formu15}
\end{align}
with some positive constant $\lambda>0$.
\end{thm}
Theorem \ref{thm11} is an extension of Hessian Sobolev inequality in \cite{Wang94class, Trudinger97symmetrization}. Using the Alexandrov-Fenchel isoperimetric inequality in \cite{GL09starshape} and the symmetrization results in \cite{Trudinger97symmetrization}, Theorem \ref{thm11} was partially proven in \cite{Gavitone08thesis, JCSW20hardy}, given that $u$ belongs to a specific function space $A_{k-1}(\Omega)$, and its sub-level set $\{x\in\Omega:u(x)<t\}$ is always $(k-1)$-convex starshaped. However, our result applies to all $k$-admissible functions in $\Phi_0^k(\Omega)$, and the technique is totally different. More precisely, our proof reduces the desired inequality to radially symmetric functions by means of a descent gradient flow, as in \cite{Wang94class}.

For an application of the Hessian Hardy-Sobolev inequality, we then turn our attention to its related variational problems, thanks to the variational structure of $S_k$ (see \cite{Wangnote}). To start with, we look at the semilinear case ($k=1$)
\begin{equation}
\left\{
\begin{array}{ll}
 -\Delta u=|x|^{2s}f(x,u)  & \text{in }\Omega, \\
 u=0  & \text{on }\partial\Omega,
\end{array}
\right.
\label{formu16}
\end{equation}
where $0\in\Omega\subset\mathbb{R}^n$ is bounded and $s>-1$. For a special case $f(x,u)=|u|^p$, the equation \eqref{formu16} is called the Hardy-H\'enon equation and has been extensively studied in the past decades, see \cite{BW06henon, GS81positive, PS12hardyhenon, Serra05henon, SWS02henon}. The existence of least-energy solutions to \eqref{formu16} follows directly by Hardy-Sobolev compact embedding via the standard variational method, when $1<p<2^*-1$. Here, the exponent $2^*$ is given by $2^*=\min(\frac{2n}{n-2},\frac{2(n+2s)}{n-2})$ for $s>-1$.

When it turns to $k\geqs2$, the situation is more complicated, since the regularity theory for $k$-Hessian equation is not that easy. In \cite{CW01variational}, Chou and Wang developed a variational theory for the Hessian equation for the first time, by studying the critical point of certain Hessian functionals. They also established appropriate uniform estimates, gradient estimates, and in particular the interior second derivatives estimates. Inspired by \cite{CW01variational}, we will study the following Dirichlet problem
\begin{equation}
\left\{
\begin{array}{ll}
 S_k(D^2u)=|x|^{2sk}f(x,u)  & \text{in }\Omega, \\
 u=0  & \text{on }\partial\Omega,
\end{array}
\right.
\label{formu17}
\end{equation}
where $0\in\Omega\subset\mathbb{R}^n$ is a bounded strictly $(k-1)$-convex domain. As the semilinear case \eqref{formu16}, we will deal with a general situation where $1\leqs k\leqs n$ and $s>-s_0$ for $s_0=\min(1,n/2k)$. Note that the equation is singular at the origin if $s<0$ while degenerate if $s>0$. We will utilize the Hessian Hardy-Sobolev inequality \eqref{formu14} when $2k<n$ and $-1<s\leqs0$, while for the other cases we use the Hessian Sobolev inequality \eqref{formu11} instead. Therefore, we extend the definition of the critical exponent $k^*$ by
\begin{align}
k^*=k^*(s)\left\{
\begin{array}{ll}
 =\frac{(k+1)(n+2sk)}{n-2k} &\text{if }2k<n\text{ and }s\leqs0,\\
 =\frac{(k+1)n}{n-2k} & \text{if }2k<n\text{ and }s>0, \\
 <\infty    & \text{if }2k=n, \\
 =\infty    & \text{if }2k>n.
\end{array}
\right.
\label{formu18}
\end{align}
In the following, we state our main results for the variational problem \eqref{formu17}. We always assume $0\in\Omega$ to be a strictly $(k-1)$-convex bounded domain with the boundary $\partial\Omega\in C^{3,1}$.
\begin{thm}\label{thm12}
Let $s>-s_0$ for $s_0=\min(1,n/2k)$ and $f^{1/k}\in C^{1,1}(\overline{\Omega}\times\mathbb{R})$. Suppose that $f(x,z)>0$ for $z<0$ and satisfies
\begin{gather}
\lim_{z\to0^-}f(x,z)/|z|^k<\lambda_1, \label{formu19}\\
\lim_{z\to-\infty}f(x,z)/|z|^k>\lambda_1, \label{formu110}
\end{gather}
and
\begin{equation}
\left\{
\begin{array}{ll}
 \lim_{z\to-\infty}f(x,z)/|z|^{k^*-1}=0  & \text{if }2k<n, \\
 \lim_{z\to-\infty}\log{f(x,z)}/|z|^{(n+2)/n}=0  & \text{if }2k=n,
\end{array}
\right.
\label{formu111}
\end{equation}
uniformly in $\overline{\Omega}$, where $\lambda_1$ is the eigenvalue of the problem \eqref{formu116}. Suppose also that there exist $\theta>0$ and $M$ large such that
\begin{align}
\int_z^0f(x,\tau)d\tau\leqs\frac{1-\theta}{k+1}|z|f(x,z)\quad\text{for } z<-M.
\label{formu112}
\end{align}
Then the problem \eqref{formu17} admits a nontrivial admissible solution $u\in\Upsilon(\Omega)$. Here, the function space $\Upsilon(\Omega)$ is given by
\begin{equation}
\left\{
\begin{array}{ll}
C^{3,\alpha}(\Omega\setminus\{0\})\cap C^{1,1}(\overline\Omega) &\text{if}\quad s\in(0,\infty),\\
C^{3,\alpha}(\Omega\setminus\{0\})\cap C^{1,1}(\overline\Omega\setminus\{0\})\cap C^{\alpha}(\overline\Omega) &\text{if}\quad s\in(-s_0,0),
\end{array}
\right.
\label{formu113}
\end{equation}
with some constant $\alpha\in(0,1)$.
\end{thm}

\begin{thm}\label{thm13}
Let $s>-s_0$ for $s_0=\min(1,n/2k)$ and $f^{1/k}\in C^{1,1}(\overline{\Omega}\times\mathbb{R})$. Suppose that $f(x,z)>0$ for $z<0$ and satisfies
\begin{gather}
\lim_{z\to0^-}f(x,z)/|z|^k>\lambda_1, \label{formu114}\\
\lim_{z\to-\infty}f(x,z)/|z|^k<\lambda_1, \label{formu115}
\end{gather}
uniformly in $\overline{\Omega}$, where $\lambda_1$ is the eigenvalue of the problem \eqref{formu116}. Then the problem \eqref{formu17} admits a nontrivial admissible solution $u\in\Upsilon(\Omega)$, where $\Upsilon(\Omega)$ is given by \eqref{formu113}.
\end{thm}
Note that when $s<0$, $u\in\Upsilon(\Omega)$ is viewed as a viscosity solution as well as a weak solution of the Dirichlet problem \eqref{formu17}; see \cite{TW97measure} and \cite{Urbas90viscosity}. Moreover, $\lambda_1$ is the (first) eigenvalue of Hessian operator $S_k$ with weights $|x|^{2sk}$. Actually, it was proved by a recent work \cite{HH25weighted} that for $1\leqs k\leqs n$ and $s>-s_0$, there exists a unique positive constant $\lambda_1=\lambda_1(n,k,s,\Omega)$ such that the eigenvalue problem
\begin{equation}
\left\{
\begin{array}{ll}
 S_k(D^2u)=\lambda|x|^{2sk}|u|^k  & \text{in }\Omega,  \\
 u=0  & \text{on }\partial\Omega,
\end{array}
\right.
\label{formu116}
\end{equation}
has a negative admissible solution $\varphi_1\in\Upsilon(\Omega)$, which is unique up to scalar multiplication. Furthermore, $\lambda_1$ satisfies the spectral feature
\begin{align}
\lambda_1=\inf_{u\in\Phi_0^k(\Omega)}\left\{\int_\Omega(-u)S_k(D^2u)dx:\int_\Omega|x|^{2sk}|u|^{k+1}dx=1\right\}.
\label{formu117}
\end{align}

As discussed in \cite{CW01variational}, Theorem \ref{thm12} and Theorem \ref{thm13} can be referred to as the superlinear case and the sublinear case, respectively. We will utilize the method in \cite{CW01variational, Tso90functional} to prove the theorems. Specifically, we make use of a descent gradient flow of the functional $J$
\begin{align*}
J(u)=\int_\Omega\frac{(-u)S_k(D^2u)}{k+1}dx-\int_\Omega F(x,u)dx,
\end{align*}
where $F(x,z)=\int_z^0|x|^{2sk}f(x,\tau)d\tau$. The Euler-Lagrange equation of $J$ is precisely \eqref{formu17}. For the sublinear case, we obtain a flow that subconverges to a minimizer of $J$. For the superlinear case, we use the underlying idea of the mountain pass lemma and derive a min-max critical point of $J$. To prove the convergence of solution, we also need the uniform a priori regularity results for \eqref{formu17}.

In the following, we briefly review the regularity results of solutions to
\begin{align}
S_k(D^2u)=f\quad\text{in }\Omega.
\label{formunew}
\end{align}
For the nondegenerate case $0<f\in C^{1,1}$, the global $C^{3,\alpha}$ regularity of solutions was first solved by Caffarelli-Nirenberg-Spruck \cite{CNS85Dirichlet} and Ivochkina \cite{Ivochkina85Dirichlet}, and was later developed by Guan \cite{Guan94subsolution} and Trudinger \cite{Trudinger95Dirichlet}. For the degenerate case $f\geqs0$, the $C^{1,1}$ regularity of solutions has been extensively studied as well. Ivochkina-Trudinger-Wang \cite{ITW04degenerate} obtained the global $C^{1,1}$ regularity under the assumption $f^{1/k}\in C^{1,1}$, which gave an alternative proof of Krylov \cite{Krylov89payoff, Krylov95nonlinear}. Given that $f^{1/(k-1)}\in C^{1,1}$, Dong \cite{Dong06degenerate} established the global $C^{1,1}$ regularity for \eqref{formunew} with homogeneous boundary data. For general boundary condition, Jiao-Wang \cite{JH24degenerate} derived the $C^{1,1}$ estimates for convex solutions when $\Omega$ is uniformly convex. For the interior $C^{1,1}$ estimates, Chou-Wang \cite{CW01variational} extended the Pogorelov estimate  \cite{Pogorelov75} to \eqref{formunew}, provided that $f\in C^{1,1}_{loc}$ is positive inside the domain.

However, the weight $|x|^{2sk}$ (or $|x|^{2s}$, $|x|^{2sk/(k-1)}$) is not differentiable at the origin for almost every $s\neq0$, so that we could not apply the above $C^{1,1}$ estimates to equation \eqref{formu17}. Instead, we will utilize the following regularity results established in \cite{HH25weighted}, for both cases $s>0$ and $s<0$. 

\begin{thm}\label{thm14}
Let $u\in C^{3,1}(\Omega)\cap C^3({\overline\Omega})$ be a $k$-admissible solution of \eqref{formu17}. Suppose that $f^{1/k}\in C^{1,1}(\overline\Omega\times\mathbb{R})$ satisfies $f(x,z)>0$ if $z<0$. Then there exists a constant $\alpha\in(0,1)$ such that 
\begin{enumerate}[\rm(i)]
\item if $-1<s<0$, then for any $\Omega'\Subset\overline\Omega \setminus\{0\}$ and $\Omega''\Subset\Omega\setminus\{0\}$,
\begin{align*}
\Vert u\Vert_{C^\alpha(\overline{\Omega})}\leqs K(\Omega),\quad
\Vert u\Vert_{C^{1,1}(\Omega')}\leqs L(\Omega'),\quad
\Vert u\Vert_{C^{3,\alpha}(\Omega'')}\leqs C(\Omega''),
\end{align*}
where $K(\Omega),L(\Omega'),C(\Omega'')$ depend additionally on 
$n,k,s,\alpha,f$ and $\Vert u\Vert_{L^\infty(\Omega)}$.
\item if $s>0$, then for any $\Omega'\Subset\Omega \setminus\{0\}$,
\begin{align*}
\Vert u\Vert_{C^{1,1}(\overline{\Omega})}\leqs\hat{K}(\Omega),\quad
\Vert u\Vert_{C^{3,\alpha}(\Omega')}\leqs\hat{L}(\Omega'),
\end{align*}
where $\hat{K}(\Omega),\hat{L}(\Omega')$ depend additionally on 
$n,k,s,\alpha,f$ and $\Vert u\Vert_{L^\infty(\Omega)}$.
\end{enumerate}
\end{thm}

We remark that the condition $f^{1/k}\in C^{1,1}(\overline{\Omega}\times\mathbb{R})$ plays a crucial role in the proof of Theorem \ref{thm14}, so as in Theorem \ref{thm12} and \ref{thm13}. For the special case $f(x,z)=|z|^p$ ($0<p<2k$ and $p\neq k$) or more general $f\in C^{1,1}(\overline{\Omega}\times\mathbb{R}^-)\cap C(\overline{\Omega\times\mathbb{R}^-})$, the local $C^{1,1}$ regularity of solutions to \eqref{formu17} is still not solved in $\Omega\setminus\{0\}$.

Finally, we introduce a nonexistence result for negative subsolutions of $\eqref{formu17}$ for the case $n>2k$ and $s\leqs-1$. This is a generalization of the semilinear case (see \cite{BC98nonexistence}).
\begin{thm}\label{thm15}
Let $\Omega\subset\mathbb{R}^n$ be a $(k-1)$-convex domain containing the origin. Suppose that $n>2k$, $s\leqs-1$ and $f(z)\in C^1(\mathbb{R}^-)$ is monotone decreasing with respect to $z$ satisfying $f(z)>0$ if $z<0$. Furthermore, for any $\varepsilon>0$, it holds that
\begin{align}
\int_{-\infty}^{-\varepsilon}f(z)^{-1/k}dz<\infty.
\label{formu118}
\end{align}
If $u\in C^0(\overline{\Omega})\cap\Phi_0^k(\Omega)$ is a viscosity subsolution of
\begin{equation}
\left\{
\begin{array}{ll}
 S_k(D^2u)=|x|^{2sk}f(u)  & \text{in }\Omega, \\
 u=0  & \text{on }\partial\Omega,
\end{array}
\right.
\label{formu119}
\end{equation}
then we have $u\equiv0$.
\end{thm}

This paper is organized as follows. In Section 2, we introduce some results of parabolic Hessian equations. In Section 3, we prove the Hessian Hardy-Sobolev inequality. In Section 4 and Section 5, we study the variational problem \eqref{formu17}, respectively, for the sublinear case and the superlinear case. Finally, we prove Theorem \ref{thm15} in Section 6.
\\[0.5em]
\noindent\textbf{Acknowledgements.} The authors are grateful to Professor Genggeng Huang for suggesting this question and for helpful discussions.

\section{Preliminaries}
In this section, we will give some preliminary results concerning parabolic Hessian equations for latter applications.

Let $\Omega$ be a strictly $(k-1)$-convex bounded domain in $\mathbb{R}^n$ with the boundary $\partial\Omega\in C^{3,1}$. Denote $Q=\Omega\times(0,\infty)$ and $Q_T=\Omega\times(0,T]$. Consider the parabolic Dirichlet problem
\begin{equation}
\left\{
\begin{array}{ll}
\mu(S_k(D^2u))-u_t=g(x,t,u) \quad\text{in }Q_T,\\
u=\phi \quad\text{on }\{t=0\}, \qquad u=0 \quad\text{on }\partial\Omega\times[0,T],
\end{array}
\right.
\label{formu21}
\end{equation}
where $\phi\in C^{3,1}(\overline\Omega)$, $g\in C^2(\overline Q_T\times\mathbb{R})$ and $\mu$ satisfies $\mu'(z)>0, \mu''(z)<0$ for all $z>0$,
\begin{align}
\mu(z)\to-\infty\text{ as }z\to0^+,\quad\mu(z)\to+\infty\text{ as }z\to+\infty,
\label{formu22}
\end{align}
and $\mu(\sigma_k(\lambda))$ is concave with respect to $\lambda$. A typical choice of $\mu$ is $\mu(z)=\log z$. But as in \cite{CW01variational}, we also use a different function $\mu$, which satisfies the additional condition
\begin{equation}
\mu(z)=\left\{
\begin{array}{ll}
  z^{1/p} & z\geqs1, \\
  \log z & z<1/2,
\end{array}
\quad\text{for some }p>k.\right.
\label{formu23}
\end{equation}
A function $u(x,t)\in C^{2,1}(Q_T)$ is said to be $k$-admissible with respect to the equation \eqref{formu21}, if $u(\cdot,t)$ is $k$-admissible for any given $t\in[0,T]$. We note that the condition \eqref{formu22} is to ensure $\sigma_k(\lambda)>0$, and thus the admissibility keeps at all time.

The following lemmas contain the a priori estimates and existence results of solutions to parabolic Hessian equations. The proof was given in \cite{CW01variational, Wang94class}. We refer the readers to \cite{Tso90functional, TW98poincare, Lieberman96parabolic} for more details on various nonlinear parabolic equations.
\begin{lemma}\label{lem21}
Suppose that $\phi\in\Phi_0^k(\Omega)$ satisfies the compatibility condition
\begin{align}
\mu(S_k(D^2\phi))=g(x,t,\phi)\quad\text{on }\partial\Omega\times\{t=0\},
\label{formu24}
\end{align}
and suppose also that there exists a positive constant $C_0$ such that
\begin{align}
g(x,t,u)\leqs C_0(1+|u|)\quad\forall(x,t,u)\in \overline{Q}_T\times\mathbb{R}.
\label{formu25}
\end{align}
Then for any $T>0$, the initial-boundary value problem \eqref{formu21} admits an admissible solution $u\in C^{3+\alpha,1+\alpha/2}(\overline{Q}_T)$ for some $\alpha\in(0,1)$. 

If $g$ is uniformly bounded, then we have the uniform estimate $\Vert u\Vert_{L^\infty(Q_T)}\leqs C$ with $C>0$ independent of $T$. Moreover, if $g$ and its derivatives up to second order are uniformly bounded, then we have $\Vert u\Vert_{C^{3+\alpha,1+\alpha/2}(\overline{Q}_T)}\leqs C'$ with $C'>0$ independent of $T$.
\end{lemma}

\begin{lemma}\label{lem22}
In addition to the hypotheses in Lemma \ref{lem21}, suppose further that $\mu$ satisfies the condition \eqref{formu23}. Then for any $k$-admissible solution $u\in C^{4,2}(\overline Q_T)$ to the problem $\eqref{formu21}$, we have for $0<t<T$,
\begin{gather}
|\nabla_xu(x,t)|\leqs C_1\Big(1+M_t^{p/k}\Big), \label{formu26}\\
|u_t(x,t)|\leqs C_2(1+M_t), \label{formu27}
\end{gather}
where $M_t=\sup_{Q_t}|u|$ and the constants $C_1$, $C_2$ depend only on $n, k, p, \phi$, $C_0$ in \eqref{formu25} and the gradient of $g$.
\end{lemma}

\section{Hessian Hardy-Sobolev Inequality}
In this section, we will prove Theorem \ref{thm11}. We first introduce the following lemma.
\begin{lemma}\label{lem31}
Suppose that $n>2k$, $-1\leqs s\leqs0$ and $k^*=k^*(s)$ given as in \eqref{formu13} and let $B_R=B_R(0)$ with some $R>0$. Then for all radially symmetric functions $u\in\Phi_0^k(B_R)$, it holds that
\begin{align*}
\Vert u\Vert_{L^{k^*}(B_R;|x|^{2sk})}\leqs C\Vert u\Vert_{\Phi_0^k(B_R)},
\end{align*}
where the constant $C$ depends only on $n,k$ and $s$.
\end{lemma}

\begin{proof}
For a radially symmetric function $u\in\Phi_0^k(B_R)$, we have by direct calculation
\begin{align*}
S_k(D^2u)=\binom{n-1}{k-1}u''(r)\Big[\frac{u'(r)}r\Big]^{k-1}+\binom{n-1}{k}\Big[\frac{u'(r)}r\Big]^{k}\quad\text{on }\{|x|=r,0<r<R\}.
\end{align*}
Then using integration by parts, we obtain
\begin{align}
\int_{B_R}(-u)S_k(D^2u)dx&=\omega_n\int_0^R(-u)\left\{\binom{n-1}{k-1} u''(r)\Big[\frac{u'(r)}r\Big]^{k-1}+\binom{n-1}{k} \Big[\frac{u'(r)}r\Big]^{k}\right\}r^{n-1}dr \nonumber\\
&=C\int_0^R(-u)\left(kr^{n-k}u''(r)[u'(r)]^{k-1}+(n-k)r^{n-k-1}[u'(r)]^k\right)dr\nonumber\\
&=C\int_0^R(-u) \partial_r(r^{n-k}[u'(r)]^k)dr\nonumber\\
&=C\int_0^Rr^{n-k}[u'(r)]^{k+1}dr,
\label{formu31}
\end{align}
where the last equality follows from $u'(0)=0, u(R)=0$. Since $S_k(D^2u)\geqs0$, we have $\partial_r(r^{n-k}[u'(r)]^k)\geqs0$ and hence $u'(r)\geqs0$ for $0<r<R$. On the other side,
\begin{align}
\int_{B_R}|x|^{2sk}|u|^{k^*}dx=\omega_n\int_0^Rr^{n-1+2sk}|u(r)|^{k^*}dr.
\label{formu32}
\end{align}
Applying Caffarelli-Kohn-Nirenberg inequality $\Vert|x|^{\beta}u\Vert_{L^p}\leqs \widetilde C\Vert|x|^{\alpha}Du\Vert_{L^q}$ (see \cite{CKN84inequality}) for dimension $N=1$ and 
\begin{align*}
q=k+1,\quad \alpha=\frac{n-k}{k+1},\quad p=k^*,\quad \text{and}\quad \beta=\frac{n-1+2sk}{k^*},
\end{align*}
we can obtain
\begin{align}
\left(\int_0^Rr^{n-1+2sk}|u(r)|^{k^*}dr\right)^{1/k^*}\leqs \widetilde C\left(\int_0^Rr^{n-k}[u'(r)]^{k+1}dr\right)^{1/(k+1)},
\label{formu33}
\end{align}
where the constant $\widetilde C$ depends only on $n,k$ and $s$. Combining \eqref{formu31}$\sim$\eqref{formu33}, we finally derive the desired result.
\end{proof}

\textbf{Proof of Theorem \ref{thm11}.} We divide the proof into three steps.
\\[0.5em]
\textbf{Step 1.} We prove Theorem \ref{thm11} holds for general $k$-admissible functions when $\Omega=B_R(0)$ for any $R>0$. Indeed, denote
\begin{gather*}
T_s=\inf\left\{\frac{\Vert u\Vert_{\Phi_0^k(B_R)}^{k+1}}{\Vert u\Vert_{L^{k^*}(B_R;|x|^{2sk})}^{k+1}}: u\in\Phi_0^k(B_R)\right\},\\
T_{s,r}=\inf\left\{\frac{\Vert u\Vert_{\Phi_0^k(B_R)}^{k+1}}{\Vert u\Vert_{L^{k^*}(B_R;|x|^{2sk})}^{k+1}}: u\in\Phi_0^k(B_R)\text{ is radial}\right\}.
\end{gather*}
By Lemma \ref{lem31}, we have $T_{s,r}\geqs c_0>0$ for some $c_0$ independent of $R$. We then claim that $T_s=T_{s,r}$. Suppose on the contrary that $T_s<T_{s,r}$. Fix a constant $\lambda\in(T_s,T_{s,r})$ and consider the functional
\begin{align}
J(u)=J(u,\Omega)=\int_\Omega\frac{(-u)S_k(D^2u)}{k+1}dx-\frac\lambda{k+1}\left(k^*\int_\Omega F(x,u)dx\right)^{(k+1)/k^*},
\label{formu34}
\end{align}
where
\begin{align*}
F(x,u)=(|x|^2+\delta^2)^{sk}\int_0^{|u|}f(t)dt,
\end{align*}
and $f$ is a smooth, positive function satisfying
\begin{equation}
f(t)=
\left\{
\begin{array}{ll}
   \delta^{k^*-1}&,|t|<\delta\\
   |t|^{k^*-1}&,2\delta<|t|<M\\
   \epsilon t^{-2}&,|t|>M+\epsilon
\end{array}
\right.,
\label{formu35}
\end{equation}
where $M>0$ is a large constant and $\delta,\epsilon>0$ are small constants. We can also assume that $f$ is monotone increasing when $\delta\leqs|t|\leqs2\delta$, and $\epsilon M^{-2}\leqs f(t)\leqs|t|^{k^*-1}$ when $M\leqs|t|\leqs M+\epsilon$. Therefore, $F$ is uniformly bounded and $J(u)$ is bounded from below. By our choice of $\lambda$, we have
\begin{gather}
\inf\{J(u):u\in\Phi_0^k(B_R)\}<-1\quad\text{if }M>>1,
\label{formu36}\\
\inf\{J(u):u\in\Phi_0^k(B_R)\text{ is radial}\}\to0\quad\text{as }\delta\to0.
\label{formu37}
\end{gather}

Due to the variational structure of $S_k$, the Euler equation of the functional $J$ can be written as
\begin{align}
S_k(D^2u)=\lambda\eta(u)(|x|^2+\delta^2)^{sk}f(u),
\label{formu38}
\end{align}
where
\begin{align}
\eta(u)=\left(k^*\int_\Omega F(x,u)dx\right)^{(k+1-k^*)/k^*}.
\label{formu39}
\end{align}
Note that $\eta(u)$ is a constant if given some $u\in\Phi_0^k(\Omega)$ . For simplicity, we denote $\psi(x,u)=\lambda\eta(u)(|x|^2+\delta^2)^{sk}f(u)$. In order to obtain a solution of \eqref{formu38}, we next consider the parabolic equation
\begin{align}
\log S_k(D^2u)-u_t=\log\psi(x,u)\quad\text{in }(x,t)\in Q:=\Omega\times(0,+\infty),
\label{formu310}
\end{align}
with the boundary condition
\begin{align*}
u(\cdot,t)=0\quad\text{on }\partial\Omega,\ \forall t\geqs0.
\end{align*}
Select the initial condition $u_0\in\Phi_0^k(\Omega)\cap C^4(\overline\Omega)$ such that
\begin{align*}
J(u_0)<\inf_{\Phi_0^k(\Omega)}J(u)+\epsilon_0<-1,
\end{align*}
by virtue of \eqref{formu36}. By a slight modification as in \cite{CW01variational, Wangnote}, we can assume that $u_0$ also satisfies the compatibility condition $S_k(D^2u_0)=\psi(x,u_0)$ on $\partial\Omega$. Notice that the equation \eqref{formu310} is a descent gradient flow of the functional $J$. Indeed, if $u(x,t)$ is a smooth solution of \eqref{formu310}, then
\begin{align}
\frac{d}{dt}J(u(\cdot,t))&=-\int_{\Omega}\big(S_k(D^2u)-\psi(x,u)\big)u_tdx \nonumber\\
&=-\int_{\Omega}\big(S_k(D^2u)-\psi(x,u)\big)\log\frac{S_k(D^2u)}{\psi(x,u)}dx\leqs0.
\label{formu311}
\end{align}
Hence, we have the a priori estimate $J(u(\cdot,t))\leqs-1$ for $t\geqs0$. Therefore, there exists a positive constant $C_0>0$ such that
\begin{align*}
k^*\int_{\Omega}F(x,u)dx\geqs C_0
\end{align*}
holds for all $t\geqs0$. By the boundedness of $F$ and \eqref{formu39}, it follows that
\begin{align*}
C_1\leqs\eta(u)\leqs C_2,
\end{align*}
where $C_1,C_2$ are positive constants independent of $t$. Then, it follows that $\log\psi(x,u)$ is uniformly bounded when $M,\epsilon,\delta>0$ are given, and thus $u$ has a uniform $L^\infty$-bound independent of $t$. Besides, using the estimate of $u_t$ in Lemma \ref{lem22}, we can obtain that $|\partial_t\eta(u)|$ is uniformly bounded. Therefore, by applying Lemma \ref{lem21} to the parabolic equation \eqref{formu310}, there exists a global solution $u(x,t)\in C^{3+\alpha,1+\alpha/2}(\overline{Q})$ satisfying $\Vert u\Vert_{C^{3+\alpha,1+\alpha/2}(\overline{Q})}\leqs C$, where $C$ might depend on $M,\varepsilon,\delta$ but not on $t$. Since \eqref{formu311} and $J$ is bounded from below, we derive a sequence $t_j\to\infty$ such that $(d/dt)J(u(\cdot,t_j))\to0$. Hence, by applying the Arzel\`a-Ascoli Theorem, we can obtain a subsequence of $\{u(\cdot,t_j)\}$ which converges to a function $\tilde u\in\Phi_0^k(\Omega)$ in $C^3(\overline\Omega)$. Note that $\tilde u$ is a solution of the elliptic equation \eqref{formu38} in $\Omega=B_R$ and $\tilde u$ satisfies $J(\tilde u)\leqs-1$.

Applying the Alexandrov's moving plane method \cite{GNN79symmetry} to the equation \eqref{formu38}, we infer that $\tilde u$ must be a radially symmetric function. Indeed, denote
\begin{align*}
\mathscr{L}(x,u,u_{ij})=S_k(D^2u)-\psi(x,u).
\end{align*}
Notice that since $\delta,\epsilon>0$, the operator $\mathscr{L}$ is $C^1$ and the equation $\mathscr{L}(x,\tilde u,\tilde u_{ij})=0$ is uniformly elliptic. By $s\leqs0$, $\mathscr{L}(x,\cdot,\cdot)$ satisfies the monotone increasing condition with respect to $|x|>0$. Hence, by the symmetric result (see Theorem 3.1 in \cite{GNN79symmetry}), we deduce that the solution $\tilde u$ is a radial function. Therefore we have
\begin{align*}
\inf\{J(u):u\in\Phi_0^k(B_R)\text{ is radial}\}\leqs-1,
\end{align*}
which yields a contradiction to \eqref{formu37} when $\delta,\epsilon$ are small. This completes the proof of our claim $T_s=T_{s,r}$.
\\[0.5em]
\textbf{Step 2.} In this step, we deal with general $(k-1)$-convex domains $\Omega$. Denote
\begin{align*}
T_s(\Omega)=\inf\left\{\frac{\Vert u\Vert_{\Phi_0^k (\Omega)}^{k+1}}{\Vert u\Vert_{L^{k^*}(\Omega;|x|^{2sk})}^{k+1}} :u\in\Phi_0^k(\Omega)\right\}.
\end{align*}
We claim that for any smooth $(k-1)$-convex areas $\Omega_1\subset\Omega_2$, it follows that $T_s(\Omega_1)\geqs T_s(\Omega_2)$. If it is not true, let $\lambda\in(T_s(\Omega_1),T_s(\Omega_2))$ be a constant and $J(u,\Omega)$ be defined as in \eqref{formu34}. Then, by our choice of $\lambda$, we have
\begin{gather}
\inf\{J(u,\Omega_1):u\in\Phi_0^k(\Omega_1)\}<-1\quad\text{if }M>>1, \label{formu312}\\
\inf\{J(u,\Omega_2):u\in\Phi_0^k(\Omega_2)\}\to0\quad\text{as }\delta\to0.  \nonumber
\end{gather}
By repeating the process in Step 1, we can derive a $k$-admissible solution $u_1\in\Phi_0^k(\Omega_1)$ to the equation \eqref{formu38} and it satisfies $J(u_1,\Omega_1)\leqs-1$. Let $R>0$ be large enough so that $\Omega_1\subset B_R(0)$ and denote
\begin{align*}
w(x)=-M-\epsilon-\frac12\epsilon^{1/2k}(R^2-|x|^2).
\end{align*}
Recall that $f(t)=\epsilon t^{-2}$ when $|t|>M+\epsilon$, and $C_1\leqs\eta(u_1)\leqs C_2$ with constants $C_1,C_2$ independent of $\epsilon$. Hence, we have $S_k(D^2w)= C\epsilon^{1/2}>\psi(x,u_1)=S_k(D^2u_1)$ in the set $\{u_1<-M-\epsilon\}$ when $\epsilon$ is sufficiently small. By applying the comparison principle, it follows that 
\begin{align}
u_1\geqs-M-\epsilon-\epsilon^{1/2k}R^2\quad\text{in }\Omega_1,
\label{formu313}
\end{align}
and thus
\begin{align}
F(x,u_1)=(|x|^2+\delta^2)^{sk}\left(\frac1{k^*}|u_1|^{k^*}+o(1)\right),
\label{formu314}
\end{align}
if $\epsilon,\delta$ are small. Therefore, by \eqref{formu39} we have
\begin{align}
\eta(u_1)=(1+o(1))\left(\int_{\Omega_1}|x|^{2sk}|u_1|^{k^*}dx \right)^{(k+1-k^*)/k^*},
\label{formu315}
\end{align}
where $o(1)\to0$ as $\epsilon,\delta\to0$.

Extend $u_1$ to $\Omega_2$ so that $u_1=0$ in $\Omega_2-\Omega_1$. Define $\phi(x)=S_k(D^2u_1)$ in $\Omega_1$ and $\phi(x)=0$ in $\Omega_2-\Omega_1$. Consider the functional
\begin{align}
E(v)=\int_{\Omega_2}(-v)\phi dx-\lambda\left(\int_{\Omega_2} |x|^{2sk}|v|^{k^*}dx\right)^{(k+1)/k^*}=\text{\Rmnum{1}}-\lambda\text{\Rmnum{2}}.
\end{align}
Claim that $E(v)$ is concave. Observe that \Rmnum{1} is linear, and thus we only need to verify that \Rmnum{2} is convex. By direct calculation of second variation, we have
\begin{align*}
\frac{d^2}{dt^2}\text{\Rmnum{2}}(u+&tv)\bigg|_{t=0}=(k+1)(k^*-1)\left(\int_{\Omega_2}|x|^{2sk}|u|^{k^*-2}|v|^2dx\right)\left(\int_{\Omega_2}|x|^{2sk}|u|^{k^*}dx\right)^{(k+1-k^*)/k^*}\\
&+(k+1)(k+1-k^*)\left(\int_{\Omega_2}|x|^{2sk}|u|^{k^*-2}uvdx\right)^2 \left(\int_{\Omega_2}|x|^{2sk}|u|^{k^*}dx\right)^{(k+1-2k^*)/k^*}.
\end{align*}
Hence by H\"older's inequality, it follows that $(d^2/dt^2)\text{\Rmnum{2}}(u+tv)\big|_{t=0}\geqs0$ for any $u$ and $v$, which implies that \Rmnum{2} is convex.

Since $u_1=0$ in $\Omega_2-\Omega_1$, we have by \eqref{formu312} and \eqref{formu314}
\begin{align*}
E(u_1)&=\int_{\Omega_1}(-u_1)S_k(D^2u_1)dx-\lambda \left(\int_{\Omega_1}|x|^{2sk}|u_1|^{k^*}dx\right)^{(k+1)/k^*}\\
&=(k+1)J(u_1,\Omega_1)+o(1)\leqs-k,
\end{align*}
when $o(1)$ is sufficiently small as $\delta,\epsilon\to0$. Consider $u_{2,m}\in\Phi_0^k(\Omega_2)$ as the solution of
\begin{align*}
S_k(D^2u)=\phi_m\quad\text{in }\Omega_2,
\end{align*}
where $\{\phi_m\}$ is a sequence of smooth positive functions which converges decreasingly to $\phi$. By the comparison principle, we have $u_{2,m}<u_1\leqs0$ in $\Omega_1$. Furthermore, $u_{2,m}$ is uniformly bounded in $C(\overline{\Omega}_2)$. Therefore, $u_2=u_{2,m}$ satisfies
\begin{align*}
E(u_2)&=\int_{\Omega_2}(-u_2)\phi dx-\lambda \left(\int_{\Omega_2}|x|^{2sk}|u_2|^{k^*}dx\right)^{(k+1)/k^*}\\
&\geqs\int_{\Omega_2}(-u_2)S_k(D^2u_2)dx-\lambda \left(\int_{\Omega_2}|x|^{2sk}|u_2|^{k^*}dx\right)^{(k+1)/k^*}+o(1)\geqs-\frac12,
\end{align*}
provided $m$ large and $\delta,\epsilon$ small enough. Here, the last inequality follows by our choice of $\lambda\in(T_s(\Omega_1),T_s(\Omega_2))$.

Denote $\rho(t)=E(u_1+t(u_2-u_1))$. Then it follows $\rho(0)=E(u_1)\leqs-k$ and $\rho(1)=E(u_2)\geqs-\frac12$. Claim that $\rho'(0)<0$. Indeed, we compute
\begin{align*}
\rho'(0)=&\int_{\Omega_1}(u_1-u_2)S_k(D^2u_1)dx \\
&-\lambda(k+1)\left(\int_{\Omega_1}|x|^{2sk}|u_1|^{k^*-1}(u_1-u_2)dx\right)\left(\int_{\Omega_1}|x|^{2sk}|u_1|^{k^*}dx\right)^{(k+1-k^*)/k^*}.
\end{align*}
Since $u_1$ solves \eqref{formu38}, by \eqref{formu35}, \eqref{formu313} and \eqref{formu315}, we have
\begin{align*}
&\int_{\Omega_1}(u_1-u_2)S_k(D^2u_1)dx\\
=&\lambda\eta(u_1) \int_{\Omega_1}(u_1-u_2)(|x|^2+\delta^2)^{sk}f(u_1)dx \\
=&\lambda(1+o(1))\left(\int_{\Omega_1}|x|^{2sk}|u_1|^{k^*-1}(u_1-u_2)dx+o(1)\right)\left(\int_{\Omega_1}|x|^{2sk} |u_1|^{k^*}dx \right)^{(k+1-k^*)/k^*} \\
<&\lambda(k+1)\left(\int_{\Omega_1}|x|^{2sk}|u_1|^{k^*-1}(u_1-u_2)dx\right)\left(\int_{\Omega_1}|x|^{2sk} |u_1|^{k^*}dx \right)^{(k+1-k^*)/k^*},
\end{align*}
provided $\delta,\epsilon>0$ sufficiently small. Hence, we obtain $\rho'(0)<0$. Since the functional $E$ is concave, we have $\rho'(t)<0$ for all $t\in[0,1]$. Thus it must follow that $\rho(1)<\rho(0)$, which leads to a contradiction.
\\[0.5em]
\textbf{Step 3.} By Step 1 and Step 2, we prove that for any $(k-1)$-convex domain, the inequality \eqref{formu14} holds for all $u\in\Phi_0^k(\Omega)$. What remains is the existence of the extremal function. Here we utilize the idea of \cite[Theorem 3.1]{GY00multiple}, and assert that when $-1<s\leqs0$, the best constant of \eqref{formu14} can be attained when $\Omega=\mathbb{R}^n$ by the function defined as \eqref{formu15}.

Indeed, it is shown by Step 2 that the best constant in \eqref{formu14} remains the same if the function $u$ is restricted in the set of all radially symmetric admissible functions. Thus, we consider the radial case. Recall that for $u=u(r)$,
\begin{gather*}
\int_{\mathbb{R}^n}(-u)S_k(D^2u)dx=C\int_0^{\infty}r^{n-k}[u'(r)]^{k+1}dr,\\
\int_{\mathbb{R}^n}|x|^{2sk}|u|^{k^*}dx=C'\int_0^{\infty}r^{n-1+2sk}|u(r)|^{k^*}dr.
\end{gather*}
To continue, we need the following lemma from Bliss \cite{Bliss30inequality}.
\begin{lemma}
Let $p_0,q_0$ be constants such that $q_0>p_0>1$. Let $f(x)$ be a real-valued nonnegative measurable function in the interval $0\leqs x<\infty$ such that the integral $J_0=\int_0^{\infty}f^{p_0}(x)dx$ is finite and given. Then the integral $g(x)=\int_0^xf(t)dt$ is finite for every $x$, and
\begin{align*}
I_0=\int_0^{\infty}g^{q_0}(x)x^{s_0-q_0}dx
\end{align*}
attains its maximum value at the functions of the form
\begin{align*}
f(x)=(\lambda x^{s_0}+1)^{-(s_0+1)/s_0},
\end{align*}
where $s_0=q_0/p_0-1$ and $\lambda$ be a positive constant.
\end{lemma}

By setting $t=r^{(2k-n)/k}$, one can directly compute
\begin{gather*}
\int_0^{\infty}r^{n-k}[u'(r)]^{k+1}dr=C_{n,k}\int_0^{\infty}|u'(t)|^{k+1}dt,\\
\int_0^{\infty}r^{n-1+2sk}|u(r)|^{k^*}dr=C_{n,k}'\int_0^{\infty}t^{-k^*k/(k+1)-1}|u(t)|^{k^*}dt.
\end{gather*}
Note that $k^*>k+1$ holds when $-1<s\leqs0$. Using the above lemma, we can deduce that if $\int_0^{\infty}|u'(t)|^{k+1}dt$ is given, then $\int_0^{\infty}t^{-k^*k/(k+1)-1}|u(t)|^{k^*}dt$ attains its maximum value when $u$ satisfies
\begin{align*}
|u'(t)|=(\lambda t^{s_0}+1)^{-(s_0+1)/s_0},
\end{align*}
where $s_0=k^*/(k+1)-1=2k(s+1)/(n-2k)$. Since $u\in\Phi_0^k(\mathbb{R}^n)$, we have $u\leqs0$ and $u(t)\big|_{t=0}=u(r)\big|_{r=\infty}=0$. Hence,
\begin{align}
u(t)=-\int_0^t|u'(\tau)|d\tau=-(\lambda+t^{-s_0})^{-1/s_0}.
\label{formu317}
\end{align}
By putting $t=r^{(2k-n)/k}$ into the equality \eqref{formu317}, we conclude that the best constant is attained at the function
\begin{align*}
u(x)=-(\lambda+|x|^{2(s+1)})^{(2k-n)/2k(s+1)}.
\end{align*}
This theorem is finally proved.  \hfill $\square$

\begin{corollary}\label{cor31}
Let $\Omega\subset\mathbb{R}^n$ be any $(k-1)$-convex bounded domain containing the origin. Suppose $1\leqs k\leqs n$, $s>-s_0$ with $s_0=\min(1,n/2k)$. Then it holds for all $u\in\Phi_0^k(\Omega)$,
\begin{align}
\Vert u\Vert_{L^p(\Omega;|x|^{2sk})}\leqs C\Vert u\Vert_{\Phi_0^k(\Omega)}\quad\text{for }p\in[1,k^*],
\label{formu318}
\end{align}
where the constant $C$ depends only on $n,k,s,\Omega$ and $p$. Here, $k^*=k^*(s)$ is given by \eqref{formu18}.
\end{corollary}
\begin{proof}
We only consider the case $s<0$. Otherwise, $|x|^{2sk}\leqs(diam(\Omega))^{2sk}<\infty$ for $s\geqs0$. Then, the inequality \eqref{formu318} follows directly by the Hessian Sobolev inequality \eqref{formu11}.

We next consider three sub-cases separately. For $2k<n$, the inequality \eqref{formu318} is an easy consequence of Theorem \ref{thm11} and H\"older's inequality, since $\Omega$ is bounded. For $2k>n$, we have $\int_\Omega|x|^{2sk}dx\leqs M<\infty$ by $s>-n/2k$, and thus for any $1\leqs p\leqs\infty$,
\begin{align*}
\left(\int_\Omega|x|^{2sk}|u|^pdx\right)^{1/p}\leqs M^{1/p}\Vert u\Vert_{L^\infty(\Omega)}\leqs C\Vert u\Vert_{\Phi_0^k(\Omega)}.
\end{align*}
Finally for $2k=n$, since $-1<s<0$, there exists a constant $\epsilon<0$ such that $-1<s+\epsilon<0$. Then, we have $\int_\Omega|x|^{2(s+\epsilon)k}dx\leqs\widetilde{M}<\infty$, and hence for any $1\leqs p<\infty$,
\begin{align*}
\left(\int_\Omega|x|^{2sk}|u|^pdx\right)^{1/p}&\leqs\left(\int_\Omega|x|^{2(s+\epsilon)k}dx\right)^{s/p(s+\epsilon)}\left(\int_\Omega|u|^{p(s+\epsilon)/\epsilon}dx\right)^{\epsilon/p(s+\epsilon)}\\
&\leqs \widetilde{M}^{s/p(s+\epsilon)}\Vert u\Vert_{L^{p(s+\epsilon)/\epsilon}(\Omega)}\leqs \widetilde{C}\Vert u\Vert_{\Phi_0^k(\Omega)}.
\end{align*}
Note that we use \eqref{formu11} to yield the last inequality. This finishes the proof.
\end{proof}

\section{The Sublinear Case}
In this section, we deal with the variational problem for the sublinear case. Before that, we introduce the $L^\infty$-estimate for solutions of
\begin{equation}
\left\{
\begin{array}{ll}
 S_k(D^2u)=(|x|^2+\delta^2)^{sk}f(x,u)  & \text{in }\Omega, \\
 u=0  & \text{on }\partial\Omega,
\end{array}
\right.
\label{formu41}
\end{equation}
where $0<\delta<1$ and $f(x,z)$ satisfies \eqref{formu115}. That is, there exist $\theta>0$, $K>0$ such that
\begin{align}
f(x,z)\leqs K+(\lambda_1-\theta)|z|^k \quad\text{for }z\leqs0.
\label{formu42}
\end{align}

\begin{lemma}\label{lem41}
Consider \eqref{formu41} where $s>-s_0$ for $s_0=\min(1,n/2k)$ and \eqref{formu42} holds. Then for any admissible solution $u$ of \eqref{formu41}, it holds
\begin{align*}
\Vert u\Vert_{L^\infty({\Omega})}\leqs M,
\end{align*}
where the constant $M>0$ depends only on $n,k,s,\Omega$ and $\theta,K$ in \eqref{formu42}.
\end{lemma}
\begin{proof}
Suppose on the contrary that there is a sequence of $\{\delta_m\}\to0$ and $\{f_m\}$ such that the equation \eqref{formu41} for $\delta=\delta_m, f=f_m$ has a solution $u_m\in\Phi_0^k(\Omega)$ satisfying
\begin{align*}
M_m=\Vert u_m\Vert_{L^\infty({\Omega})}\to\infty\quad\text{as }m\to\infty.
\end{align*}
Denote $v_m=u_m/M_m$. Then $v_m$ satisfies
\begin{align*}
S_k(D^2v_m)=(|x|^2+\delta_m^2)^{sk}M_m^{-k}f_m(x,M_mv_m).
\end{align*}
Using \eqref{formu42} and Theorem \ref{thm14}, $v_m$ subconverges to a nonzero function $v\in\Upsilon(\Omega)$, which is a supersolution of
\begin{align}
S_k(D^2u)=(\lambda_1-\theta)|x|^{2sk}|u|^k.
\label{formu43}
\end{align}
On the other hand, let $a>1$ be sufficiently large such that $w=a\varphi_1<v$ in $\Omega$, where $\varphi_1\in\Upsilon(\Omega)$ is the eigenfunction of \eqref{formu116}. Hence, $w$ and $v$ are, respectively, a subsolution and a supersolution of \eqref{formu43}. By the method of subsolution and supersolution, we obtain an admissible solution $\varphi^*\in\Upsilon(\Omega)$ of \eqref{formu43} with $w\leqs\varphi^*\leqs v$. This contradicts the uniqueness result for the eigenvalue problem \eqref{formu116}, see \cite{HH25weighted}. This completes the proof.
\end{proof}

Recall the functional $J$
\begin{align*}
J(u)=\int_\Omega\frac{(-u)S_k(D^2u)}{k+1}dx-\int_\Omega F(x,u)dx,
\end{align*}
where $F(x,z)=\int_z^0|x|^{2sk}f(x,\tau)d\tau$. Assume that $f$ satisfies \eqref{formu115}, then there exist $\theta_1>0$ and $K_1>0$ such that
\begin{align*}
F(x,z)\leqs K_1+\frac{(1-\theta_1)\lambda_1}{k+1}|x|^{2sk}|z|^{k+1}.
\end{align*}
According to \eqref{formu117}, we obtain a lower bound estimate for $J$
\begin{align}
J(u)\geqs\frac{\theta_1}{k+1}\int_\Omega(-u)S_k(D^2u)dx-K_1|\Omega|.
\label{formu44}
\end{align}

The main result in this section is as follows, which contains Theorem \ref{thm13}.
\begin{thm}\label{thm41}
Let $\Omega$ be a strictly $(k-1)$-convex bounded domain containing the origin with the boundary $\partial\Omega\in C^{3,1}$. Let $s>-s_0$ for $s_0=\min(1,n/2k)$ and $f^{1/k}\in C^{1,1}(\overline{\Omega}\times\mathbb{R})$. Suppose $f(x,z)>0$ for $z<0$ such that \eqref{formu114} and \eqref{formu115} holds uniformly in $\overline{\Omega}$. Then the problem \eqref{formu17} has a nontrivial admissible solution $u\in\Upsilon(\Omega)$, which is a minimizer of the functional $J$ over $\Phi_0^k(\Omega)$.
\end{thm}
\begin{proof}
For $m\in\mathbb{N}$, let $\hat{f}_m$ such that $\hat{f}_m^{1/k}\in C^{1,1}(\overline{\Omega}\times\mathbb{R})$ and
\begin{align*}
\hat{f}_m(x,z)=f(x,z) \quad\text{for }|z|<m,\quad\hat{f}_m(x,z)=f(x,-2m) \quad\text{for }|z|>2m.
\end{align*}
Let $f_m^{1/k}=\hat{f}_m^{1/k}+1/m$. Consider the functional $J=J_m$ 
\begin{align*}
J_m(u)=\int_\Omega\frac{(-u)S_k(D^2u)}{k+1}dx-\int_\Omega F_m(x,u)dx,
\end{align*}
where $F_m(x,z)=\int_z^0(|x|^2+m^{-2})^{sk} f_m(x,\tau)d\tau$. Similar to \eqref{formu44}, we have $J_m(u)\geqs-K_2$ with some $K_2$ independent of $m$ large. On the other hand, since \eqref{formu114} we can take $a>0$ sufficiently small such that for the eigenfunction $\varphi_1$ of \eqref{formu116},
\begin{align*}
f_m(x,z)\geqs f(x,z)\geqs (\lambda_1+\tilde\theta)|z|^k\quad\text{holds in }\{z:a\varphi_1<z<0\},
\end{align*}
with a constant $\tilde\theta>0$. Hence, for $m$ sufficiently large,
\begin{align*}
J_m(a\varphi_1)&=\int_\Omega\frac{(-a\varphi_1) S_k(aD^2\varphi_1)}{k+1}dx-\int_\Omega F_m(x,a\varphi_1)dx \\
&\leqs\int_\Omega\frac{\lambda_1|x|^{2sk}|a\varphi_1|^{k+1}}{k+1}dx-\int_\Omega\frac{\lambda_1+\tilde\theta}{k+1}(|x|^2+m^{-2})^{sk}|a\varphi_1|^{k+1}dx \\
&\leqs\int_\Omega\frac{-\tilde\theta|x|^{2sk}|a\varphi_1|^{k+1}}{k+1}dx+o(1)\leqs -c_0<0,
\end{align*}
for some $c_0$ independent of $m$. This illustrates that $\inf_{\Phi_0^k(\Omega)}J_m\leqs-c_0<0$ for $m$ large.

Consider the parabolic equation
\begin{equation}
\left\{
\begin{array}{ll}
\log S_k(D^2u)-u_t=\log\psi_m(x,u) \quad\text{in }Q:=\Omega\times(0,\infty),  \\
u=u_0\quad\text{on }\{t=0\},\qquad u=0\quad\text{on }\partial\Omega\times[0,\infty),
\end{array}
\right.
\label{formu45}
\end{equation}
where $\psi_m(x,u)=(|x|^2+m^{-2})^{sk}f_m(x,u)$ and $u_0\in\Phi_0^k(\Omega)$ satisfies $J_m(u_0)<\inf_{\Phi_0^k(\Omega)}J_m+\varepsilon_m$ with $0<\varepsilon_m\to0^+$ as $m\to\infty$. Using a similar statement as in Section 3, we can assume that $u_0$ satisfies the compatibility condition $S_k(D^2u_0)=\psi_m(x,u_0)$ on $\partial\Omega\times\{t=0\}$. Notice that $\log\psi_m$ and its derivatives up to second order are uniformly bounded, which might depend on $m$ but not on time $t$. Thus, applying Lemma \ref{lem21} to \eqref{formu45}, we obtain a global admissible solution $u$ satisfying $\Vert u\Vert_{C^{3+\alpha,1+\alpha/2}(\overline{Q})}\leqs C$ with $C$ independent of $t$. Moreover, similar to \eqref{formu311}, $u(\cdot,t)$ is a descent gradient flow of $J_m$, namely
\begin{align*}
\frac{d}{dt}J_m(u(\cdot,t))=-\int_\Omega\big(S_k(D^2u)-\psi_m(x,u)\big)\log\frac{S_k(D^2u)}{\psi_m(x,u)}dx\leqs0.
\end{align*}
Since $J_m$ is bounded from below, there exists a sequence $\{t_j\}$ tending to $+\infty$ such that $(d/dt)J_m(u(\cdot,t_j))\to0$. Hence we can extract a subsequence of $\{u(\cdot,t_j)\}$ which converges in $C^3(\overline{\Omega})$ to a function $u_m\in\Phi_0^k(\Omega)$. Then, $u_m$ is a solution of \eqref{formu41} with $\delta=m^{-1}$ and $f$ replaced by $f_m$, and it follows $\inf_{\Phi_0^k(\Omega)}J_m\leqs J_m(u_m)\leqs\inf_{\Phi_0^k(\Omega)}J_m+\varepsilon_m$.

Since $f_m$ satisfies \eqref{formu42} with uniform $K>0$ and $\theta>0$, then by Lemma \ref{lem41} we have the uniform estimate $\Vert u_m\Vert_{L^\infty(\Omega)}\leqs M$. Therefore, by applying Theorem \ref{thm14} to $\{u_m\}$, we can obtain a subsequence of $\{u_m\}$ which converges to a solution $u\in\Upsilon(\Omega)$ of the problem \eqref{formu17}. Furthermore, by the definition of $J$ and $J_m$, we have by $\varepsilon_m\to0^+$,
\begin{align*}
-K_2\leqs\inf_{\Phi_0^k(\Omega)}J=\lim_{m\to\infty} \inf_{\Phi_0^k(\Omega)}J_m=\lim_{m\to\infty}J_m(u_m)=J(u)\leqs-c_0<0.
\end{align*}
Hence, we conclude that $u\neq0$ is a minimizer of the functional $J$ over $\Phi_0^k(\Omega)$.
\end{proof}

\section{The Superlinear Case}
In this section, we prove Theorem \ref{thm12}. Here we apply a slight modification to the proof in \cite{CW01variational}, so that the approximation $f_\delta$ of $f$ is suitable to Theorem \ref{thm14}.
\\[-0.5em]

\textbf{Proof of Theorem \ref{thm12}.} For clarity, we divide the proof into six steps.
\\
\textbf{Step 1.} Let us first assume $f$ satisfies a growth condition stronger than \eqref{formu111},
\begin{align}
\limsup_{z\to-\infty}\frac{f(x,z)}{|z|^p}<+\infty\quad\text{uniformly in }\overline{\Omega},
\label{formu51}
\end{align}
where $k<p<k^*-1$ is to be determined in Step 4.

For $0<\delta<1$, define $f_\delta$ by $f_\delta^{1/k}=f^{1/k}+\delta$. For convenience, we also denote $f_0=f$. Then, consider the approximation problem of \eqref{formu17}
\begin{equation}
\left\{
\begin{array}{ll}
 S_k(D^2u)=(|x|^2+\delta^2)^{sk}f_\delta(x,u)  & \text{in }\Omega, \\
 u=0  & \text{on }\partial\Omega,
\end{array}
\right.
\label{formu52}
\end{equation}
and its related functional $J_\delta$,
\begin{align*}
J_\delta(u)=\int_\Omega\frac{(-u)S_k(D^2u)}{k+1}dx-\int_\Omega F_\delta(x,u)dx,
\end{align*}
where $F_\delta(x,z)=\int_z^0(|x|^2+\delta^2)^{sk}f_\delta(x,\tau) d\tau$. Let $\tilde u_1\equiv0$ and $\tilde u_2=a\varphi_1$ for $a>1$ large. Then $J_0(\tilde u_1)=0$, and by \eqref{formu110} there exist $\theta_1>0$ and $C_1>0$ such that
\begin{align*}
J_0(\tilde u_2)&=\int_\Omega\frac{(-a\varphi_1) S_k(aD^2\varphi_1)}{k+1}dx-\int_\Omega F_0(x,a\varphi_1)dx \\
&\leqs\int_\Omega\frac{\lambda_1|x|^{2sk}|a\varphi_1|^{k+1}}{k+1}dx-\int_\Omega\frac{(\lambda_1+\theta_1) |x|^{2sk}|a\varphi_1|^{k+1}}{k+1}dx+C_1 \\
&\leqs-\frac{\theta_1}{k+1}\int_\Omega|x|^{2sk} |a\varphi_1|^{k+1}dx+C_1<-1,
\end{align*}
provided $a>1$ sufficiently large. Let $u_1$ and $u_2$ be smooth $k$-admissible functions close to $\tilde u_1$ and $\tilde u_2$, respectively, such that $S_k(D^2u_i)>0$ in $\overline{\Omega}$ for both $i=1,2$, and that $|J_0(u_1)|$ be sufficiently small and $J_0(u_2)<-1$. By the definition of $J_\delta$, we can obtain a small $\delta_0>0$ so that for $0<\delta<\delta_0$, $|J_\delta(u_1)|$ is small and $J_\delta(u_2)<-1$.

Denote by $\Gamma$ the set of paths in $\Phi_0^k(\Omega)$ connecting $u_1$ and $u_2$ continuously, namely,
\begin{align*}
\Gamma=\{\gamma\in C([0,1],&\Phi_0^k(\Omega)\cap C^{3,1}(\overline{\Omega})): \\
&\gamma(0)=u_1,\gamma(1)=u_2,S_k(D^2\gamma(\tau))>0\text{ in }\overline{\Omega}\text{ for }\tau\in[0,1]\}.
\end{align*}
Define the min-max value $c_\delta$ by
\begin{align}
c_\delta=\inf_{\gamma\in\Gamma}\sup_{\tau\in[0,1]}J_\delta(\gamma(\tau)).
\label{formu53}
\end{align}
Then it is easy to check that $c_0\geqs\limsup_{\delta\to0}c_\delta$.

Next, we derive a positive lower bound for $c_0$. Indeed, by \eqref{formu19} and \eqref{formu51}, there exist $\theta_2>0$ and $C>0$ such that
\begin{align}
f(x,z)\leqs\lambda_1(1-\theta_2)|z|^k+C|z|^p,
\label{formu54}
\end{align}
Since $k<p<k^*-1$, we have
\begin{align*}
J_0(u)&=\frac{1}{k+1}\Vert u\Vert_{\Phi_0^k(\Omega)}^{k+1}-\int_\Omega|x|^{2sk}\int_u^0f(x,z)dzdx \\
&\geqs\frac{1}{k+1}\Vert u\Vert_{\Phi_0^k(\Omega)}^{k+1}-\int_\Omega\left(\frac{\lambda_1(1-\theta_2)}{k+1}|x|^{2sk} |u|^{k+1}+\frac{C}{p+1}|x|^{2sk}|u|^{p+1}\right)dx \\
&\geqs\frac{\theta_2}{k+1}\Vert u\Vert_{\Phi_0^k(\Omega)}^{k+1}-C'\Vert u\Vert_{\Phi_0^k(\Omega)}^{p+1},
\end{align*}
where the last inequality follows by \eqref{formu117} and \eqref{formu318}. Hence, by taking a small $\sigma>0$, we have
\begin{align}
J_0(u)\geqs\frac{\theta_2}{2(k+1)}\sigma^{k+1}>0,\quad\text{for all }u\in\Phi_0^k(\Omega)\text{ with }\Vert u\Vert_{\Phi_0^k(\Omega)}=\sigma.
\label{formu55}
\end{align}

Let $u_1$ be sufficiently close to $\tilde u_1=0$ satisfying $\Vert u_1\Vert_{\Phi_0^k(\Omega)}<\sigma/2$. Then for any $\gamma\in\Gamma$, there must exist a $\tau_0\in(0,1)$ such that $\Vert\gamma(\tau_0)\Vert_{\Phi_0^k(\Omega)}=\sigma$. By \eqref{formu55}, we have
\begin{align*}
c_0\geqs\epsilon_0:=\frac{\theta_2}{2(k+1)}\sigma^{k+1}>0.
\end{align*}
Using a similar argument as above, we can show that there exists a $\delta_0>0$ small enough such that $c_\delta\geqs\epsilon_0/2>0$ holds for any $0<\delta<\delta_0$. We finally note that $c_0$ and $c_\delta$ have a uniform upper bound $C^*$ independent of $\delta$. Indeed, selecting a bounded path $\gamma\in\Gamma$, this follows by setting $C^*=\sup_{\tau\in[0,1]} \Vert\gamma(\tau)\Vert_{\Phi_0^k(\Omega)}^{k+1}<\infty$.
\\[0.5em]
\textbf{Step 2.} In this step, we prove that $c_\delta$ is a critical value of $J_\delta$ and there exists an admissible solution $u_\delta$ of \eqref{formu52} with $J_\delta(u_\delta)=c_\delta$.

For any $0<\varepsilon<\epsilon_0/4$, choose a $\gamma\in\Gamma$ that satisfies $\sup_{\tau\in[0,1]}J_\delta(\gamma(\tau))\leqs c_\delta+\varepsilon$. Then consider the parabolic problem
\begin{equation}
\left\{
\begin{array}{ll}
\mu(S_k(D^2u))-u_t=\mu(\psi_\delta(x,u)) \quad\text{in }Q:=\Omega\times(0,\infty),  \\
u=\gamma(\tau)\quad\text{on }\{t=0\},\qquad u=0\quad\text{on }\partial\Omega\times[0,\infty),
\end{array}
\right.
\label{formu56}
\end{equation}
where $\psi_\delta(x,u)=(|x|^2+\delta^2)^{sk}f_\delta(x,u)$ and $\mu$ is specified as in Section 2. More precisely, $\mu$ satisfies \eqref{formu23} with the exponent $p$ given in \eqref{formu51}, and it holds
\begin{align}
(a-b)(\mu(a)-\mu(b))\geqs(a-b)(a^{1/p}-b^{1/p})\quad\text{for }a,b>0.
\label{formu57}
\end{align}
We further assume that for every $\tau\in[0,1]$, $\gamma(\tau)$ satisfies the compatibility condition \eqref{formu24} $S_k(D^2\gamma(\tau))=\psi_\delta(x,\gamma(\tau))$ on $\partial\Omega\times\{t=0\}$. Observe that $\mu(\psi_\delta)$ satisfies the condition \eqref{formu25}, according to \eqref{formu51}. Hence, according to Lemma \ref{lem21}, there exists a global admissible solution $u^\tau(x,t)$ to \eqref{formu56}, for every $\tau\in[0,1]$.

Denote $\gamma^{t_0}(\tau)=u^\tau(\cdot,t_0)$ for any given $t_0\in(0,\infty)$. Then as discussed in \cite{Tso90functional}, we see that $\gamma^{t_0}$ is a path in $\Phi_0^k$. Let
\begin{align*}
\gamma_1=\{u^\tau(\cdot,t):\tau=0,0\leqs t\leqs t_0\},\quad \gamma_2=\{u^\tau(\cdot,t):\tau=1,0\leqs t\leqs t_0\}.
\end{align*}
Connect $\gamma_1,\gamma^{t_0}$ and $\gamma_2$ together so as to form a path in $\Gamma$, and denote it by $\tilde\gamma^{t_0}$.

Similar to \eqref{formu311}, we have
\begin{align}
\frac{d}{dt}J_\delta(u^\tau(\cdot,t))&=-\int_{\Omega}\big(S_k(D^2u^\tau)-\psi_\delta(x,u^\tau)\big)\partial_tu^\tau dx \nonumber\\
&=-\int_{\Omega}\big(S_k(D^2u^\tau)-\psi_\delta(x,u^\tau)\big)
\big(\mu(S_k(D^2u^\tau))-\mu(\psi_\delta(x,u^\tau))\big)dx\leqs0.
\label{formu58}
\end{align}
This illustrates that $u^\tau(\cdot,t)$ is a descent gradient flow of the functional $J_\delta$.

Set $I_t=\{\tau\in[0,1]:J_\delta(\gamma^t(\tau))\geqs c_\delta-\varepsilon\}$. Obviously $I_t$ is a closed subset of $[0,1]$, and $I_t\subset I_{t'}$ holds for any $t\geqs t'$, by virtue of \eqref{formu58}. Let $I_\infty=\cap_{t\geqs0}I_t$. We claim that $I_\infty$ is not empty. If it is not true, there exists a $t_0\in(0,\infty)$ such that $I_{t_0}=\varnothing$, i.e., $J_\delta(\gamma^{t_0}(\tau))<c_\delta-\varepsilon$ for all $\tau\in[0,1]$. Then we have $J_\delta(u)\leqs c_\delta-\varepsilon$ for all $u\in\tilde\gamma^{t_0}$, which contradicts the definition of $c_\delta$. Thus, there has at least one point $\tau_0\in I_\infty$.

In the next two steps, we will prove $|u^{\tau_0}(x,t)|\leqs M_0<\infty$ for all $t\geqs0$. Using Lemma \ref{lem21} again, we have the estimates $\Vert u^{\tau_0}(\cdot,t)\Vert_{C^{3+\alpha,1+\alpha/2}(\overline{Q})}\leqs C$ uniformly for $t\in(0,\infty)$.
Since $J_\delta(u^{\tau_0}(\cdot,t))$ is bounded from below, there exists a sequence $\{t_j\}\to\infty$ such that $(d/dt)J_\delta(u^{\tau_0}(\cdot,t_j))$ tends to $0$. Then we can extract a subsequence of $\{u^{\tau_0}(\cdot,t_j)\}$, which converges to a nontrivial solution $u_\delta$ of \eqref{formu52} with $c_\delta-\varepsilon\leqs J_\delta(u_\delta)\leqs c_\delta+\varepsilon$.
\\[0.5em]
\textbf{Step 3.} In the following, we will write $u^{\tau_0}$ as $u$, dropping the superscript $\tau_0$ for brevity. Recall that $c_\delta-\varepsilon\leqs J_\delta(u(\cdot,t))\leqs c_\delta+\varepsilon$ for all time $t$. Then denote the set
\begin{align*}
K^0=\left\{t\in(0,\infty):\frac{d}{dt}J_\delta(u(\cdot,t))<-\varepsilon\right\},
\end{align*}
and we have $\mathrm{mes}(K^0)\leqs2$. In this step, we will show that for any $t\not\in K^0$,
\begin{gather}
\int_\Omega(-u(\cdot,t))S_k(D^2u(\cdot,t))dx\leqs C, \label{formu59}\\
\int_\Omega F_\delta(\cdot,u(\cdot,t)) dx\leqs C, \label{formu510}
\end{gather}
where the constant $C$ is independent of $t, \varepsilon$ and $\delta$.

For $t\not\in K^0$, by \eqref{formu57} and \eqref{formu58} we have
\begin{align*}
\int_\Omega\big(S_k(D^2u)-\psi_\delta(x,u)\big)\big(S_k^{1/p}(D^2u)-\psi_\delta^{1/p}(x,u)\big)dx\leqs-\frac{d}{dt}J_\delta(u(\cdot,t))\leqs\varepsilon.
\end{align*}
Denote $\mathcal{A}=S_k^{1/p}(D^2u)$ and $\mathcal{B}=\psi_\delta^{1/p}(x,u)$. We obtain
\begin{align*}
\int_\Omega|\mathcal{A}-\mathcal{B}|^{p+1}dx\leqs C\int_\Omega|\mathcal{A}^p-\mathcal{B}^p||\mathcal{A}-\mathcal{B}|dx\leqs C\varepsilon.
\end{align*}
Hence, we have
\begin{align}
&\left|\int_\Omega u(\mathcal{A}^p-\mathcal{B}^p)dx\right| \nonumber\\
&\leqs C\int_\Omega|u||\mathcal{A}-\mathcal{B}|(\mathcal{A}^{p-1}+\mathcal{B}^{p-1})dx  \nonumber\\
&\leqs C\left(\int_\Omega|\mathcal{A}-\mathcal{B}|^{p+1}dx \right)^{\frac{1}{p+1}} \left(\int_\Omega|u|^{p+1}dx \right)^{\frac{1}{p(p+1)}} \left(\int_\Omega|u|(\mathcal{A}^p +\mathcal{B}^p)dx\right)^{\frac{p-1}{p}} \nonumber\\
&\leqs C\varepsilon^{1/(p+1)}\Vert u\Vert_{L^{p+1}(\Omega)}^{1/p} \left[\left(\int_\Omega|u|\mathcal{A}^pdx\right)^{\frac{p-1}{p}}+\left(\int_\Omega|u|\mathcal{B}^pdx\right)^{\frac{p-1}{p}}\right].  \label{formu511}
\end{align}
On the other hand, since $f_\delta^{1/k}=f^{1/k}+\delta$, we have
\begin{align}
F_\delta(x,u)&\leqs(|x|^2+\delta^2)^{sk}\int_u^0\Big[(1+\nu)f(x,z)+C_\nu\delta\Big]dz \nonumber\\
&\leqs(|x|^2+\delta^2)^{sk}\left[\frac{(1-\theta)(1+\nu)}{k+1}|u|f(x,u)+C_\nu\delta|u|+C\right] \nonumber\\
&\leqs \frac{(1-\theta)(1+\nu)}{k+1}|u|\psi_\delta(x,u)+(|x|^2+\delta^2)^{sk}\Big[C_\nu\delta|u|+C\Big],
\label{formu512}
\end{align}
where the second inequality follows by \eqref{formu112}. We can take $\nu>0$ sufficiently small so that $(1-\theta)(1+\nu)=1-\theta'$ for some $\theta'>\theta/2$. We still denote $\theta'$ by $\theta$ for simplicity. Therefore, 
\begin{align*}
J_\delta&(u)=\int_\Omega\frac{(-u)S_k(D^2u)}{k+1}dx-\int_\Omega F_\delta(x,u)dx  \\
&\geqs\int_\Omega\frac{(-u)S_k(D^2u)}{k+1}dx-\int_\Omega\left\{\frac{1-\theta}{k+1}|u|\psi_\delta(x,u)+(|x|^2+\delta^2)^{sk}\Big[C_\nu\delta|u|+C\Big]\right\}dx  \\
&\geqs\frac{1}{k+1}\int_\Omega|u|(\mathcal{A}^p-\mathcal{B}^p)dx+\frac{\theta}{k+1}\int_\Omega|u| \psi_\delta(x,u)dx-\int_\Omega(|x|^2+\delta^2)^{sk} \Big[C_\nu\delta|u|+C\Big]dx.
\end{align*}
Then using the inequalities \eqref{formu318} and \eqref{formu511}, we obtain (note that $k<p<k^*-1$)
\begin{align*}
&\int_\Omega|u|\psi_\delta(x,u)dx  \\
&\leqs C\left|\int_\Omega|u|(\mathcal{A}^p-\mathcal{B}^p)dx\right|+C\delta\int_\Omega (|x|^2+\delta^2)^{sk}|u|dx+C(1+J_\delta(u))  \\
&\leqs C\varepsilon^{1/(p+1)}\Vert u\Vert_{L^{p+1}(\Omega)}^{1/p} \left[\left(\int_\Omega|u|\mathcal{A}^pdx\right)^{\frac{p-1}{p}}+\left(\int_\Omega|u|\mathcal{B}^pdx\right)^{\frac{p-1}{p}}\right]+C\delta\Vert u\Vert_{\Phi_0^k(\Omega)}+C  \\
&\leqs C\varepsilon^{1/(p+1)}\left[\int_\Omega|u|\mathcal{A}^pdx +\left(\int_\Omega|u|\mathcal{A}^pdx\right)^{\frac{1}{p}}\left(\int_\Omega|u|\mathcal{B}^pdx\right)^{\frac{p-1}{p}}\right]+C\delta\Vert u\Vert_{\Phi_0^k(\Omega)}+C,
\end{align*}
where we use $J_\delta(u)\leqs c_\delta+\varepsilon\leqs C^*$ from Step 1. Hence, we have
\begin{align}
\int_\Omega|u|\psi_\delta(x,u)dx\leqs C\beta_{\varepsilon,\delta}\int_\Omega|u|\mathcal{A}^pdx+C,
\label{formu513}
\end{align}
where $\beta_{\varepsilon,\delta}\to0$ as $\varepsilon,\delta\to0$ and $C$ is independent of $\varepsilon,\delta$. Inserting the estimates \eqref{formu512}, \eqref{formu513} into $J_\delta(u)\leqs C^*$ and choosing $\varepsilon,\delta>0$ sufficiently small, we finally obtain \eqref{formu59} and \eqref{formu510} for any $t\not\in K^0$.
\\[0.5em]
\textbf{Step 4.} In this step, we prove the uniform $L^\infty$-bound for $u(\cdot,t)$ for $t\geqs0$. Denote $M_t=\sup_{\Omega}|u(\cdot,t)|$. Suppose on the contrary that there exists a sequence $\{t_j\}\to\infty$, such that $M_{t_j}\to\infty$ and
\begin{align}
M_{t_j}\geqs M_t\quad\text{for all }t\in[0,t_j].
\label{formu514}
\end{align}
Using the estimate \eqref{formu27} of $u_t$ and the assumption \eqref{formu514}, we have
\begin{align*}
M_{t}\geqs M_{t_j}e^{C_1(t-t_j)}\quad\text{for }t\leqs t_j.
\end{align*}
In particular, $M_t\geqs CM_{t_j}$ for $t\in[t_j-2,t_j]$. Since $K^0$ has a measure less than $2$, we can choose $t_j'\in[t_j-2,t_j]$ but $\not\in K^0$ and it satisfies $M_t\leqs CM_{t_j'}$ for all $t<t_j'$. Denote $M_j=M_{t_j'}$ for simplicity and we have $M_j\to\infty$ as $j\to\infty$.

Suppose that the maximum $M_j$ of $|u(\cdot,t_j')|$ is attained at a point $x_j\in\Omega$. By the global gradient estimate \eqref{formu26}, we have
\begin{align}
|u(x,t_j')|\geqs\frac{1}{2}M_j\quad \text{for }x\in B_{r_j}(x_j),
\label{formu515}
\end{align}
where $r_j=C_0M_j^\beta$ for $C_0>0$ independent of $j$, and $\beta=1-p/k=(k-p)/k$.

By \eqref{formu59} and the Hessian Sobolev inequality \eqref{formu11}, we have
\begin{align*}
\Vert u(\cdot,t_j')\Vert_{L^q(B_{r_j}(x_j))}\leqs\Vert u(\cdot,t_j')\Vert_{L^q(\Omega)}\leqs C\Vert u(\cdot,t_j')\Vert_{\Phi_0^k(\Omega)}\leqs C,
\end{align*}
where $1\leqs q\leqs k^{\star}$, and $k^{\star}$ is the Sobolev exponent of $S_k$.
On the other hand, by \eqref{formu515}
\begin{align*}
\Vert u(\cdot,t_j')\Vert_{L^q(B_{r_j}(x_j))}^q\geqs CM_j^qr_j^n\geqs CM_j^{q+n\beta}.
\end{align*}
Assuming $q+n\beta>0$ for a moment, we then reach a contradiction when $M_j\to\infty$. Therefore, we obtain $\sup_{\Omega}|u(\cdot,t)|\leqs M_0<+\infty$ for all $t\geqs0$. 

Finally, we need to select suitable $p$ and $q$ that satisfy all hypotheses. When $k\geqs n/2$, we fix any $p\in(k,+\infty)$ and then let $q$ be large enough so that $q+n\beta=q+n(k-p)/k>0$. When $k<n/2$, we take $q=k^{\star}=n(k+1)/(n-2k)$ and choose a $p$ satsifying
\begin{align*}
k<p<\min\left\{k^*-1,\frac{k(n+1-k)}{n-2k}\right\}=\min\left\{\frac{k(n+2+2s+2sk)}{n-2k},\frac{k(n+1-k)}{n-2k}\right\}.
\end{align*}
By direct computation, we have $q+n\beta>0$. We note that in the proof we always select a $p$ satisfying the above condition, and we will deal with the general situation in the last step.
\\[0.5em]
\textbf{Step 5.} By Step $2\sim4$, we have obtained a solution $u_\delta$ of \eqref{formu52} with $c_\delta-\varepsilon\leqs J_\delta(u_\delta)\leqs c_\delta+\varepsilon$, and it satisfies
\begin{align}
\int_\Omega(-u_\delta)S_k(D^2u_\delta)dx\leqs C,
\label{formu516}
\end{align}
for some $C>0$ independent of $\delta,\varepsilon$. In this step, we show that $M_\delta=\sup_\Omega|u_\delta(x)|$ is uniformly bounded and thus $u_\delta$ subconverges to a solution $u$ of \eqref{formu17}.

When $k>n/2$, this is an easy consequence by combining \eqref{formu516} and the Hessian Sobolev embedding \eqref{formu11}. When $k=n/2$, by H\"older's inequality we obtain
\begin{align*}
\big\Vert\psi_\delta(x,u_\delta)\big\Vert_{L^\sigma(\Omega)}\leqs C\big\Vert|x|^{2sk}\big\Vert_{L^{\sigma^2}(\Omega)}\Big\Vert f_\delta(x,u_\delta)\Big\Vert_{L^{\frac {\sigma^2}{\sigma-1}}(\Omega)}\leqs C\Big\Vert f_\delta(x,u_\delta)\Big\Vert_ {L^{\frac{\sigma^2}{\sigma-1}}(\Omega)},
\end{align*}
where we set $\sigma>1$ sufficiently close to 1 so that $s\sigma^2>-1$. We claim that the last term has a uniform bound independent of $\delta,\varepsilon$. Indeed, by \eqref{formu111} we have
\begin{align*}
\log f_\delta(x,u_\delta)\leqs C_{\epsilon}+\epsilon|u_\delta|^{(n+2)/n},
\end{align*}
for any $\epsilon>0$. Denote $q=\sigma^2/(\sigma-1)$. By Moser-Trudinger inequality \eqref{formu12}, we obtain
\begin{align*}
\Vert f_\delta(x,u_\delta)\Vert_{L^q(\Omega)}^q&\leqs \int_\Omega\exp\left[{q\big(C_\epsilon+\epsilon|u_\delta|^{(n+2)/n}\big)}\right]dx  \\
&\leqs\widetilde C_\epsilon\int_\Omega\exp\bigg[\alpha_n\Big( \frac{|u_\delta|}{\Vert u_\delta\Vert_{\Phi_0^k(\Omega)}}\Big) ^{(n+2)/n}\bigg]dx\leqs C\widetilde C_\epsilon,
\end{align*}
where $\epsilon>0$ is taken sufficiently small such that $q\epsilon\leqs\alpha_n\Vert u_\delta\Vert_{\Phi_0^k(\Omega)}^{-(n+2)/n}$ holds uniformly for $\delta,\varepsilon>0$, by virtue of \eqref{formu516}. Hence, we have the uniform estimate $\Vert\psi_\delta(x,u_\delta)\Vert_{L^\sigma(\Omega)}\leqs C$. Note that $\sigma>1$, then by applying the $L^\infty$-estimate in \cite[Theorem 2.1]{CW01variational} to \eqref{formu52}, we derive a uniform $L^\infty$-bound for $u_\delta$.

When $k<n/2$, we make use of a rescaling method. We only consider the case $-1<s<0$; the case $s\geqs0$ can be settled by a similar argument. Suppose that $M_\delta$ tends to $\infty$. Denote
\begin{align*}
v_\delta(y)=M_\delta^{-1}u_\delta(R_\delta^{-1}y)\quad \text{in }D_\delta:=\{y:R_\delta^{-1}y\in\Omega\},
\end{align*}
where $R_\delta=M_\delta^{\beta_0}$ with $\beta_0=(k^*-1-k)/2k(1+s)$. Then $-1\leqs v_\delta\leqs0$ and $\inf_{D_\delta}v_\delta=-1$. Moreover, $v_\delta$ satisfies
\begin{align}
S_k(D^2v_\delta)=\widetilde\psi_\delta(y,v_\delta):=M_\delta^{-(k^*-1)}(|y|^2+R_\delta^2\delta^2)^{sk}f_\delta(M_\delta v_\delta)\quad \text{in }D_\delta.
\label{foru517}
\end{align}
By a direct calculation, it follows that
\begin{align*}
\int_{D_\delta}(-v_\delta)S_k(D^2v_\delta)dy=\int_\Omega(-u_\delta)S_k(D^2u_\delta)dx.
\end{align*}
Using the Hessian Sobolev inequality \eqref{formu11}, we have by \eqref{formu516}
\begin{align*}
\left(\int_{D_\delta}|v_\delta|^{k^\star}dy\right)^{1/k^\star}\leqs C_{n,k}\left(\int_{D_\delta}(-v_\delta)S_k(D^2v_\delta)dy \right)^{1/(k+1)}\leqs C.
\end{align*}
Denote $\widetilde D_\delta=D_\delta\cap\{v_\delta\leqs-\frac12\}$. Then we have $\text{mes}(\widetilde D_\delta)\leqs C_1$. Let $\sigma>n/2k$ be such that $0>2\sigma sk>-n$. Then it follows
\begin{align*}
\int_{\widetilde D_\delta}|y|^{2\sigma sk}dy\leqs \int_{\widetilde D_\delta\cap B_1}|y|^{2\sigma sk}dy+\int_{\widetilde D_\delta\cap B_1^c}|y|^{2\sigma sk}dy\leqs C.
\end{align*}
Therefore, we infer that $\Vert\widetilde\psi_\delta(y,v_\delta)\Vert_{L^\sigma(\widetilde D_\delta)}$ tends to $0$ as $M_\delta\to\infty$, since $|z|^{-(k^*-1)}f(x,z)$ converges uniformly to $0$ as $|z|\to\infty$. Applying the $L^\infty$-estimate in \cite[Theorem 2.1]{CW01variational} to $v_\delta+\frac12$ over the domain $\widetilde D_\delta$, we obtain $v_\delta\geqs-\frac34$ when $M_\delta$ is large enough, which leads to a contradiction to $\inf_{D_\delta}v_\delta=-1$.

Hence, $M_\delta$ is uniformly bounded. Then letting $\delta,\varepsilon\to0$ and using Theorem \ref{thm14}, we can obtain a subsequence of $\{u_\delta\}$, which converges to a solution $u\in\Upsilon(\Omega)$ of \eqref{formu17}.

Furthermore, by \eqref{formu516} we can only consider the path $\gamma\in\Gamma$ such that $\Vert\gamma(\cdot)\Vert_{\Phi_0^k(\Omega)}\leqs C$ with some constant $C$ independent of $\delta,\varepsilon$. Then by the inequality \eqref{formu318}, we have
\begin{align*}
\lim_{\delta\to0}J_\delta(\gamma(\tau))=J_0(\gamma(\tau))\quad \text{uniformly for }\tau\in[0,1].
\end{align*}
Thus, by the definition \eqref{formu53} of $c_\delta$, we derive $c_0=\lim_{\delta\to0}c_\delta$. Hence, we have
\begin{align*}
J_0(u)=\lim_{\delta,\varepsilon\to0}J_\delta(u_\delta)=\lim_{\delta\to0}c_\delta=c_0.
\end{align*}
That is, solution $u$ achieves the min-max critical value of $J_0$.
\\[0.5em]
\textbf{Step 6.} Finally, we remove assumption \eqref{formu51} by constructing feasible approximation functions. Given $f$ satisfying \eqref{formu19}$\sim$\eqref{formu112}, denote
\begin{align*}
f_m(x,z)=\left\{
\begin{array}{ll}
 f(x,z)  & \text{if }z>-m, \\
 d_m(x)|z|^p  & \text{if }z<-m,
\end{array}
\right.
\end{align*}
where $p$ is specified as in Step 4 and $d_m(x)=m^{-p}f(x,-m)$. We can also slightly modify $f_m$ at $z=-m$ if necessary. Clearly, $f_m$ satisfies \eqref{formu19}$\sim$\eqref{formu111} and \eqref{formu51}. For \eqref{formu112}, if $z<-m$,
\begin{align}
\int_z^0f_m(x,\tau)d\tau&=\int_{-m}^0f(x,\tau)d\tau+\int_z^{-m}d_m(x)|\tau|^pd\tau  \nonumber\\
&\leqs\frac{1-\theta}{k+1}mf(x,-m)+\frac{1}{p+1}d_m(x)(|z|^{p+1}-m^{p+1})  \nonumber\\
&\leqs\frac{1-\theta'}{k+1}|z|f_m(x,z),
\label{formu519}
\end{align}
with $\theta'>0$ depending only on $\theta$ and $p$. Then we can obtain a function $u_m\in\Upsilon(\Omega)$ that solves
\begin{align*}
S_k(D^2u)=|x|^{2sk}f_m(x,u):=\psi_m(x,u)\quad \text{in }\Omega.
\end{align*}
Moreover, $J_m(u_m)=c_m$, where
\begin{align}
J_m(u)=\int_\Omega\frac{(-u)S_k(D^2u)}{k+1}dx-\int_\Omega F_m(x,u)dx,
\label{formu520}
\end{align}
$F_m(x,z)=\int_z^0|x|^{2sk}f_m(x,\tau)d\tau$, and $c_m$ is the critical value of $J_m$ defined as in \eqref{formu53}. By a similar argument as in Step 1, we have uniform lower and upper positive bounds for $c_m$, i.e., $c'\leqs c_m\leqs c''$ for $c',c''>0$ independent of $m$.

We first show that there exists $C>0$ independent of $m$ such that
\begin{align}
\int_\Omega(-u_m)S_k(D^2u_m)dx\leqs C.
\label{formu521}
\end{align}
Indeed, combining \eqref{formu520} and
\begin{align*}
\int_\Omega\Big[(-u_m)S_k(D^2u_m)+u_m\psi_m(x,u_m)\Big]dx=0,
\end{align*}
we obtain
\begin{align*}
\int_\Omega\Big[\frac{(-u_m)\psi_m(x,u_m)}{k+1}-F_m(x,u_m)\Big]dx\leqs c''.
\end{align*}
Hence, we have by \eqref{formu519}
\begin{align*}
\int_\Omega|u_m|\psi_m(x,u_m)dx\leqs C.
\end{align*}
Inserting this inequality into \eqref{formu520}, we derive the uniform estimate \eqref{formu521}.

With \eqref{formu521} at hand, we can repeat the argument in Step 5 to obtain the uniform boundedness of $|u_m|$. The case $k\geqs n/2$ is completely the same. We only need to verify when $k<n/2$,
\begin{align}
|z|^{-(k^*-1)}f_m(x,z)\to0\quad\text{as }z\to-\infty,\text{ uniformly for }m\in\mathbb{N}.
\label{formu522}
\end{align}
Indeed, for any $\varepsilon>0$ there exists a constant $M^*>0$ such that $|z|^{-(k^*-1)}f(x,z)\leqs\varepsilon$ for $z<-M^*$. Then for $m>M^*$, we consider two cases separately. For $-m<z<-M^*$, $f_m(x,z)=f(x,z)$ and thus $|z|^{-(k^*-1)}f_m(x,z)\leqs\varepsilon$ holds. For $z\leqs-m$,
\begin{align*}
|z|^{-(k^*-1)}f_m(x,z)&=|z|^{-(k^*-1)}m^{-p}f(x,-m)|z|^p  \\
&\leqs\left(\frac{m}{|z|}\right)^{k^*-1-p}m^{-(k^*-1)}f(x,-m)\leqs\varepsilon,
\end{align*}
by our assumption $p<k^*-1$. We actually prove \eqref{formu522}.

By Theorem \ref{thm14}, there exists a subsequence of $\{u_m\}$ converging to a solution $u\in\Upsilon(\Omega)$ of the problem \eqref{formu17}, which satisfies $J_0(u)=c_0>0$. We finally complete the proof of Theorem \ref{thm12}.    \hfill $\square$

\section{Nonexistence Results}
In this section, we utilize the idea of \cite{BC98nonexistence} to prove Theorem \ref{thm15}.
\\[-0.5em]

\textbf{Proof of Theorem \ref{thm15}.}
Assume that $B_{2\eta}(0)\subset\Omega$ for some $0<\eta<1$. Suppose on the contrary that there exists a nonzero viscosity subsolution $u\in C^0(\overline{\Omega})\cap\Phi_0^k(\Omega)$ to \eqref{formu119}. Then $u$ is subharmonic and achieves its maximum $u=0$ on the boundary $\partial\Omega$. By strong maximum principle, we have $u(B_{\eta})\leqs-\varepsilon_0$ for some $\varepsilon_0>0$. Set
\begin{align*}
\phi(z)=\int_{-\varepsilon_0}^zf(t)^{-1/k}dt,
\end{align*}
for $z\leqs-\varepsilon_0$. Then we have
\begin{gather}
\phi'(z)=f(z)^{-1/k}>0, \label{formu61} \\
\phi''(z)=-\frac1kf(z)^{-1-1/k}f'(z)\geqs0,  \nonumber
\end{gather}
since $f(z)>0$, $f'(z)\leqs0$ for $z<0$. Note that $\phi(u)\leqs0$ in $B_{\eta}$ and $\phi(u)>-\infty$ by \eqref{formu118}. 

Denote $\tilde u=\phi(u)$. We claim that $\tilde u$ is a viscosity subsolution of
\begin{equation}
\left\{
\begin{array}{ll}
S_k(D^2w)=|x|^{2sk} & \text{in }B_{\eta}(0), \\
w=0 & \text{on }\partial B_\eta(0).
\end{array}
\right.
\label{formu62}
\end{equation}
Indeed, consider arbitrary function $\tilde v\in C^2(\mathcal{N)}$ for any open $\mathcal{N}\subset B_\eta(0)$. Assume that $\tilde v-\tilde u$ attains its local maximum at $x_0\in\mathcal{N}$. Without loss of generality, we can also assume $\tilde v(x_0)=\tilde u(x_0)\leqs0$ and $\tilde v\leqs\tilde u$ in $\mathcal{N}$. Since $\phi$ is a strictly increasing function of $C^2$, we have $v:=\phi^{-1}(\tilde v)\in C^2(\mathcal{N})$ satisfies $v(x_0)=u(x_0)\leqs-\varepsilon_0$ and $v\leqs u$ in $\mathcal{N}$. Since $u$ is a viscosity subsolution of \eqref{formu119}, we obtain (see \cite{Urbas90viscosity})
\begin{align*}
S_k(D^2v)\geqs |x|^{2sk}f(u)\quad \text{at }x=x_0.
\end{align*}
On the other side, by direct calculation,
\begin{align*}
\frac{\partial^2\tilde v}{\partial x_i\partial x_j}=\phi''(v)\frac{\partial v}{\partial x_i}\frac{\partial v}{\partial x_j}+\phi'(v)\frac{\partial^2v}{\partial x_i\partial x_j},
\end{align*}
and hence by $\phi''\geqs0$, it holds at $x=x_0$
\begin{align*}
S_k(D^2\tilde v)\geqs[\phi'(v)]^kS_k(D^2v)\geqs[\phi'(v)]^k|x|^{2sk}f(u)=|x|^{2sk},
\end{align*}
where the last equality follows by \eqref{formu61} and $v(x_0)=u(x_0)$. By the arbitrariness of $\tilde v$, we deduce that $\tilde u$ is a subsolution of \eqref{formu62}.

Next, let $w$ be the admissible solution of the Dirichlet problem
\begin{equation}
\left\{
\begin{array}{ll}
  S_k(D^2w)=(|x|^2+\delta^2)^{-k}&\text{in }B_{\eta}(0),\\
  w=0&\text{on }\partial B_{\eta}(0),
\end{array}
\right.
\label{formu63}
\end{equation}
where $\delta>0$ is a small constant. Since $\eta<1$ and $s\leqs-1$, we can obtain $w\geqs\tilde u$ by comparison principle. By moving plane method \cite{GNN79symmetry}, we derive that $w$ is radial symmetric. Thus, \eqref{formu63} falls into an ODE problem:
\begin{equation*}
\left\{
\begin{array}{l}
\partial_r(r^{n-k}[w'(r)]^k)=C_{n,k}r^{n-1}(r^2+\delta^2)^{-k},\\
 w'(0)=0, \quad  w(\eta)=0.
\end{array}
\right.
\end{equation*}
Integrating the equation from $0$ to $\rho\in(0,\eta)$, we obtain
\begin{align*}
\rho^{n-k}[w'(\rho)]^k=C_{n,k}\int_0^\rho\frac{r^{n-1}}{(r^2+\delta^2)^k}dr.
\end{align*}
Then for $\rho\geqs2\delta$, we have
\begin{align*}
\rho^{n-k}[w'(\rho)]^k\geqs C_{n,k}\int_{\delta}^\rho\frac{r^{n-1}}{(r^2+\delta^2)^k}dr\geqs C'_{n,k}(\rho^{n-2k}-\delta^{n-2k})\geqs c_{n,k}\rho^{n-2k}.
\end{align*}
Thus, $w'(\rho)\geqs c_{n,k}\rho^{-1}$ for $\rho\geqs2\delta$. Note that $w(\eta)=0$, then integrating from $2\delta$ to $\eta$ yields
\begin{align*}
w(2\delta)\leqs c_{n,k}(\log2\delta-\log\eta).  
\end{align*}
Since $\tilde u\leqs w$, we reach a contradiction to the boundedness of $\tilde u$ when $\delta$ is taken sufficiently small. This finishes the proof.            \hfill $\square$

\small
~\\
\noindent\textbf{Data Availability } No datasets were generated or analyzed during the current study.
\\[0.5em]
\noindent\textbf{Conflict of interest } On behalf of all authors, the corresponding author declares that there is no conflict of interest.

\bibliographystyle{plain}
\bibliography{cite}

\medskip\medskip
{\em Address and E-mail:}
\medskip\medskip

{\em Rongxun He}

{\em School of Mathematical Sciences, Fudan University}

{\em rxhe24@m.fudan.edu.cn}

\medskip\medskip

{\em Wei Ke}

{\em School of Mathematics, Southwest Minzu University}

{\em wke25@swun.edu.cn}

\end{document}